\documentclass[final]{siamltex}
\usepackage{cite}
\usepackage{algorithm}
\usepackage{algorithmic}
\usepackage{amsmath,amssymb}
\usepackage{bm}
\usepackage{booktabs}
\usepackage{epsfig, epstopdf}
\usepackage{float}
\usepackage{graphics}
\usepackage{graphicx}
\usepackage{multido}
\usepackage{psfrag}
\usepackage{pstricks}
\usepackage{pst-all}
\usepackage{pst-poly}
\usepackage{setspace}
\usepackage{subfigure}
\usepackage{xypic}
\usepackage{color}
\usepackage{mathdots}

\newcommand{\C}{\mathbb{C}}
\newcommand{\R}{\mathbb{R}}
\newcommand{\coord}[1]{\mathbf{#1}}

\newcommand{\Adj}{\mathbf{A}}
\DeclareMathOperator{\Vm}{\mathbf{V}}
\DeclareMathOperator{\Um}{\mathbf{U}}

\DeclareMathOperator{\Eig}{\mathbf{\Lambda}}
\DeclareMathOperator{\ZeroMatrix}{\mathbf{0}}

\DeclareMathOperator{\Id}{\mathbf{I}}
\DeclareMathOperator{\Bm}{\mathbf{B}}
\DeclareMathOperator{\Cm}{\mathbf{C}}
\DeclareMathOperator{\Dm}{\mathbf{D}}
\DeclareMathOperator{\Hm}{\mathbf{H}}
\DeclareMathOperator{\Pm}{\mathbf{P}}
\DeclareMathOperator{\Qm}{\mathbf{Q}}
\DeclareMathOperator{\Tm}{\mathbf{T}}

\DeclareMathOperator{\DSTI}{\mathbf{DST-1}}

\newcommand{\mypar}[1]{{\bf #1.}}

\title{Eigendecomposition of Block Tridiagonal Matrices\thanks{This work was supported in part by AFOSR grant FA95501210087.}}

\author{
Aliaksei Sandryhaila
  \thanks{Department of Electrical and Computer Engineering,
  Carnegie Mellon University, Pittsburgh, PA 15213
  (asandryh@andrew.cmu.edu). }
\and
Jos\'{e} M.~F. Moura
  \thanks{Department of Electrical and Computer Engineering,
  Carnegie Mellon University, Pittsburgh, PA 15213
  (moura@ece.cmu.edu).}
}

\begin{document}

\maketitle

%%%%%%%%%%%%%%%%%%%%%%%%%%%%%%%%%%%%%%%%%%%%%%%%%%%%%%%%%%%%%%%%%
%%%%%%%%%%%%%%%%%%%%%%%%%%%%%%%%%%%%%%%%%%%%%%%%%%%%%%%%%%%%%%%%%
%%%%%%%%%%%%%%%%%%%%%%%%%%%%%%%%%%%%%%%%%%%%%%%%%%%%%%%%%%%%%%%%%

\begin{abstract}
%\comment{Abstract is at most 250 words.}
Block tridiagonal matrices arise in applied mathematics, physics, and signal processing.
Many applications require knowledge of eigenvalues and eigenvectors of block tridiagonal matrices, which can be prohibitively expensive for large matrix sizes.
In this paper, we address the problem of the eigendecomposition of block tridiagonal matrices by studying a connection
between their eigenvalues and zeros of appropriate matrix polynomials. We use this connection with matrix polynomials
to derive a closed-form expression for the eigenvectors of block tridiagonal matrices, which eliminates the need for their direct calculation and can lead to a faster calculation of eigenvalues.
We also demonstrate with an example that our work can lead to fast algorithms for the eigenvector expansion for block tridiagonal matrices.
\end{abstract}

\begin{keywords}
Block tridiagonal matrix, eigendecomposition, eigenvalue, eigenvector, matrix polynomials, polynomial recursion, fast algorithm, eigenvector expansion, spectral analysis, graphs.
\end{keywords}

\begin{AMS}
Primary: 15A18, 65F15, 15B99.
Secondary: 65T50.
\end{AMS}

\pagestyle{myheadings}
\thispagestyle{plain}
\markboth{A.~SANDRYHAILA AND J.~M.~F.~MOURA}{EIGENDECOMPOSITION OF BLOCK TRIDIAGONAL MATRICES}

%%%%%%%%%%%%%%%%%%%%%%%%%%%%%%%%%%%%%%%%%%%%%%%%%%%%%%%%%%%%%%%%%
%%%%%%%%%%%%%%%%%%%%%%%%%%%%%%%%%%%%%%%%%%%%%%%%%%%%%%%%%%%%%%%%%
%%%%%%%%%%%%%%%%%%%%%%%%%%%%%%%%%%%%%%%%%%%%%%%%%%%%%%%%%%%%%%%%%

\section{Introduction}

We consider the problem of calculating exact eigenvalues and eigenvectors of an arbitrary $N\times N$ block tridiagonal matrix
\begin{equation}
\label{eq:block_tridiagonal_form}
\Adj = \begin{pmatrix}
\Bm_0 & \Dm_0 \\
\Cm_0 & \Bm_1 & \ddots \\
& \ddots & \ddots & \Dm_{L-2} \\
&& \Cm_{L-2} & \Bm_{L-1}
\end{pmatrix}.
\end{equation}
Here, the blocks $\Bm_n, \Cm_n, \Dm_n \in \C^{K\times K}$ are arbitrary complex $K\times K$ matrices. The matrix size $N$ then is divisible by the block size $K$, and we define $L=N/K$.
The only requirement we place on the matrix $\Adj$ in~\eqref{eq:block_tridiagonal_form} is that the blocks $\Dm_n$ are non-singular matrices, i.e., for all $n\geq 0$,
\begin{equation}
\label{eq:nonsingular_D}
\det\Dm_n \neq 0.
\end{equation}

Block tridiagonal matrices find applications in multiple areas. They arise in the analysis of random walks and birth-and-death processes~\cite{Dette:06,Grunbaum:09},
discretized transport problem simulations and electronic structure calculations~\cite{Kramer:93, Casati:94, Petersen:08},
scattering theory~\cite{Iida:90}, computational fluid dynamics~\cite{Anderson:95}, telecommunications~\cite{Lu:97}, signal processing~\cite{Moura:92, Balram:93, Asif:99, Kavcic:00, Asif:05}, data learning~\cite{Sandryhaila:13,Sandryhaila:13d}, and high performance computing~\cite{Polizzi:06}, among many others.

Many applications of block tridiagonal matrices require the calculation of their eigenvalues and eigenvectors. This problem has been extensively studied in the literature, although primarily for symmetric block tridiagonal matrices. A widely used direct approach is based on the iterative conjugation of the block tridiagonal matrix with sparse matrices until it is reduced to a tridiagonal form~\cite{Bischof:00}, followed by the solution of the eigendecomposition problem for the resulting tridiagonal matrix~\cite{Cuppen:81}. In general, this algorithm is expensive and numerically unstable, but more efficient and stable implementations have also been proposed~\cite{Gu:94, Gu:95}. Another approach to the direct calculation of eigenvectors and eigenvalues of tridiagonal matrices uses the relation between tridiagonal matrices and orthogonal polynomials~\cite{Askey:87,Gautschi:04,Sandryhaila:12}.

As an alternative to the direct approach, approximate eigendecomposition of special types of block tridiagonal matrices has been studied in~\cite{Gansterer:02,Gansterer:03}.
By approximating eigenvalues and eigenvectors, rather than computing them precisely, these methods can offer higher computational efficiency.

In this paper, we study the eigendecomposition of block tridiagonal matrices~\eqref{eq:block_tridiagonal_form} that satisfy condition~\eqref{eq:nonsingular_D}. We extend the connection between symmetric block tridiagonal matrices and orthogonal matrix polynomials proposed in~\cite{Duran:96,Dette:06,Grunbaum:09} to general, non-symmetric matrices of the form~\eqref{eq:block_tridiagonal_form}. Using this relation, we demonstrate that the eigenvalues of block tridiagonal matrices are the zeros of the determinants of appropriately constructed matrix polynomials. We construct a closed-form expression for the eigenvectors of block tridiagonal matrices that is simpler than the direct calculation of eigenvectors of the $N\times N$ matrix $\Adj$ in~\eqref{eq:block_tridiagonal_form} and instead involves only the calculation of null-space bases for $K\times K$ matrices. Since in most applications $K\ll N$, the proposed expression significantly reduces the cost of eigenvector computation for block tridiagonal matrices.

As an example, we study the spectral analysis of ``spider'' graphs, i.e., the eigendecomposition of their adjacency matrices, which possess the block tridiagonal structure~\eqref{eq:block_tridiagonal_form}. We determine the corresponding eigenvalues and eigenvectors and demonstrate that the closed-form expression for the eigenvectors leads to a fast computational algorithm for the vector expansion in the eigenvector basis of ``spider'' graphs.

%%%%%%%%%%%%%%%%%%%%%%%%%%%%%%%%%%%%%%%%%%%%%%%%%%%%%%%%%%%%%%%%%
%%%%%%%%%%%%%%%%%%%%%%%%%%%%%%%%%%%%%%%%%%%%%%%%%%%%%%%%%%%%%%%%%
%%%%%%%%%%%%%%%%%%%%%%%%%%%%%%%%%%%%%%%%%%%%%%%%%%%%%%%%%%%%%%%%%
\section{Matrix polynomials generated by block tridiagonal matrices}

Consider the blocks $\Bm_n$, $\Cm_n$, and $\Dm_n$ of the block tridiagonal matrix~\eqref{eq:block_tridiagonal_form}. Define the family of $K\times K$ matrix polynomials $\Pm_n(x)$
that satisfy the relation
\begin{equation}
\label{eq:three_term_relation}
x\cdot \Pm_n(x) = \Cm_{n-1} \Pm_{n-1}(x) + \Bm_n \Pm_n(x) + \Dm_n \Pm_{n+1}(x)
\end{equation}
with initial conditions $\Pm_{-1}(x) = \ZeroMatrix_K$ and $\Pm_0(x)=\Id_K$, respectively, the $K\times K$ zero matrix and the $K\times K$ identity matrix.

Rewrite the relation~\eqref{eq:three_term_relation} as the recurrence
\begin{equation}
\label{eq:recurrence}
\Pm_{n+1}(x) = \Dm_n^{-1} \left( x\cdot \Pm_n(x) - \Bm_n \Pm_n(x) - \Cm_{n-1} \Pm_{n-1}(x) \right).
\end{equation}
The non-singularity condition~\eqref{eq:nonsingular_D} ensures that~\eqref{eq:recurrence} and~\eqref{eq:three_term_relation} are well-defined for any block tridiagonal matrix~\eqref{eq:block_tridiagonal_form}. Since in this paper we study polynomials $\Pm_0(x)$, $\ldots$, $\Pm_L(x)$, we assume that the block $\Dm_{L-1}$ is defined and satisfies the condition~\eqref{eq:nonsingular_D}; for simplicity, we can assume that $\Dm_{L-1}=\Id_K$.\footnote{This assumption does not compromise the generality of our results, since in this paper we are only interested in the roots of the determinant of the matrix polynomial $\Pm_L(x)$. As follows from~\eqref{eq:recurrence} and the non-singularity condition~\eqref{eq:nonsingular_D}, the roots of $\det\Pm_L(x)$ are not affected by the value of $\det \Dm_{L-1}$.}

The matrix polynomials $\Pm_n(x)$ possess a number of useful properties.
As follows from~\eqref{eq:recurrence}, each element of $\Pm_n(x)$ is a complex-coefficient polynomial of degree $n$ in the variable $x$. Since $\Pm_n(x)$ is a $K\times K$ matrix, its determinant is a polynomial of degree $Kn$:
\begin{equation}
\label{eq:degree_determinant}
\deg \det \Pm_n(x) = Kn.
\end{equation}
We use these properties to derive the expressions for the eigenvalues and eigenvectors of block tridiagonal matrices~\eqref{eq:block_tridiagonal_form}.

\mypar{Remark}
Matrix polynomials generated by the relation~\eqref{eq:recurrence} can also possess an \emph{orthogonality} property. Although we do not use this property in our work, the orthogonality of matrix polynomials has been extensively studied before and we briefly overview it here.

If the block tridiagonal matrix $\Adj$ in~\eqref{eq:block_tridiagonal_form} is Hermitian, i.e., $\Adj=\Adj^H$, which means that $\Bm_n=\Bm_n^H$ and $\Cm_n^H = \Dm_n$, then there exists a
real interval $\mathcal{I}\subset\R$ and a $K\times K$ matrix weight function $\coord{W}(x)$, such that the polynomials $\Pm_n(x)$ are orthogonal over $\mathcal{I}$
with respect to $\coord{W}(x)$~\cite{Krein:49,Duran:96,Duran:05b,Grunbaum:08}:
\begin{equation}
\label{eq:orthogonality_OMP}
\int_{\mathcal{I}}\Pm_n(x)\coord{W}(x)\Pm_m^T(x)dx = \delta_{n-m}\Id_K.
\end{equation}
The orthogonality property~\eqref{eq:orthogonality_OMP} also holds for some non-Hermitian block tridiagonal matrices $\Adj$ that satisfy special conditions~\cite{Dette:06}.

%%%%%%%%%%%%%%%%%%%%%%%%%%%%%%%%%%%%%%%%%%%%%%%%%%%%%%%%%%%%%%%%%
%%%%%%%%%%%%%%%%%%%%%%%%%%%%%%%%%%%%%%%%%%%%%%%%%%%%%%%%%%%%%%%%%
%%%%%%%%%%%%%%%%%%%%%%%%%%%%%%%%%%%%%%%%%%%%%%%%%%%%%%%%%%%%%%%%%

\section{Eigenstructure of block tridiagonal matrices}
\label{sec:Eigendecomposition}

In this section, we demonstrate that the eigenvalues of the block tridiagonal matrix $\Adj$ in~\eqref{eq:block_tridiagonal_form} are closely related to the zeros of the matrix polynomials generated by the recurrence~\eqref{eq:recurrence} and derive a closed-form expression for the eigenvectors of $\Adj$.

%%%%%%%%%%%%%%%%%%%%%%%%%%%%%%%%%%%%%%%%%%%%%%%%%%%%%%%%%%%%%%%%%
\subsection{Eigenvalues and eigenvectors}

Following the notation in~\cite{Duran:96}, we call the roots of the determinant $\det \Pm_n(x)$ of a matrix polynomial $\Pm_n(x)$ the \emph{zeros} of $\Pm_n(x)$; i.e., $\lambda$ is a zero of $\Pm_n(x)$ if  $\det \Pm_n(\lambda)=0$. The following theorem shows that the eigenvalues of the block tridiagonal matrix $\Adj$ in~\eqref{eq:block_tridiagonal_form} coincide with the zeros of the matrix polynomial $\Pm_{L}(x)$ generated by the recurrence~\eqref{eq:recurrence}. This theorem also establishes a general form of the corresponding eigenvectors. The proof of the theorem follows closely the proof of Lemma~2.1 in~\cite{Duran:96} that considers symmetric block tridiagonal matrices and orthogonal polynomials.

\begin{theorem}
\label{thm:eigenvalue_is_root}
Consider an arbitrary block tridiagonal  matrix $\Adj$ of the form~\eqref{eq:block_tridiagonal_form} satisfying~\eqref{eq:nonsingular_D} and the corresponding matrix polynomials $\Pm_n(x)$ generated by the recurrence~\eqref{eq:recurrence}. Then $\coord{v}$ is an eigenvector of $\Adj$ that corresponds to an eigenvalue $\lambda$ if and only if (iff) $\lambda$ is a zero of the matrix polynomial $\Pm_L(x)$ and the vector $v$ has the form
\begin{equation}
\label{eq:eigenvector_via_P}
\coord{v} = \begin{pmatrix} \Pm_0(\lambda) \\ \vdots \\ \Pm_{L-2}(\lambda) \\ \Pm_{L-1}(\lambda) \end{pmatrix}\coord{u},
\end{equation}
where $\coord{u}\in\C^K$ is a vector from the null-space of the scalar matrix $\Pm_{L}(\lambda)$, i.e., the vector $u$ satisfies $\Pm_{L}(\lambda)\coord{u}=0 $ .
\end{theorem}
\begin{proof}
Write the eigenvector $\coord{v}$ in the block form
$$
\coord{v} = \begin{pmatrix}
\coord{v}_0 \\ \vdots \\ \coord{v}_{L-1}
\end{pmatrix},
$$
where each block $\coord{v}_n\in\C^K$ is a vector of length $K$. As follows from~\eqref{eq:block_tridiagonal_form}, the relation $\Adj\coord{v}=\lambda\coord{v}$ is equivalent to
the set of equations
\begin{eqnarray}
\label{eq:equation1}
\Bm_0\coord{v}_0 + \Dm_0\coord{v}_1 &=& \lambda \coord{v}_0, \\
\nonumber
\Cm_0\coord{v}_0 + \Bm_1\coord{v}_1 + \Dm_1\coord{v}_2 &=& \lambda \coord{v}_1, \\
\nonumber
& \vdots \\
\label{eq:equation2}
\Cm_{L-2}\coord{v}_{L-2} + \Bm_{L-1}\coord{v}_{L-1} &=& \lambda \coord{v}_{L-1}.
\end{eqnarray}

Since $\Pm_0(x)=\Id_K$, we can write $\coord{v}_0 = \Pm_0(\lambda) \coord{v}_0$.
According to the recurrence~\eqref{eq:recurrence} and the non-singularity condition~\eqref{eq:nonsingular_D}, the equation~\eqref{eq:equation1} is equivalent to
$$
\coord{v}_1 = \Dm_0^{-1}(\lambda\Id_K - \Bm_0)\Pm_0(\lambda)\coord{v}_0 = \Pm_1(\lambda)\coord{v}_0.
$$
Continuing this substitution recursively, we obtain
\begin{equation}
\label{eq:vn_is_Pn}
\coord{v}_n = \Pm_n(\lambda)\coord{v}_0
\end{equation}
for $0\leq n < L$.

Substituting $n=L-1$ in~\eqref{eq:vn_is_Pn} and using equation~\eqref{eq:equation2} and the three-term relation~\eqref{eq:three_term_relation}, we obtain
$$
\Pm_{L}(\lambda) \coord{v}_0 = 0,
$$
which is equivalent to $\coord{v}_0$ belonging to the null-space of the matrix $\Pm_{L}(\lambda)$.

Finally, since $\coord{v}_0$ is not a zero vector, the null-space of $\Pm_{L}(\lambda)$ is non-trivial. This is possible iff $\det\Pm_{L}(\lambda)=0$,
or equivalently, iff $\lambda$ is a root of $\det\Pm_{L}(x)$.

By substituting $\coord{u}$ for $\coord{v}_0$, we obtain the relation~\eqref{eq:eigenvector_via_P}.
\end{proof}

As noted above, Theorem~\ref{thm:eigenvalue_is_root} does not require the matrix~\eqref{eq:block_tridiagonal_form} to be symmetric or the matrix polynomials $\Pm_n(x)$ to be orthogonal, and extends the result obtained in~\cite{Duran:96}. The theorem shows that the three-term recurrence relation~\eqref{eq:recurrence} is sufficient for establishing a bijective relation from the eigenvalues and eigenvectors of an arbitrary block tridiagonal matrix to the zeros of the matrix polynomials and their corresponding null-spaces.

%%%%%%%%%%%%%%%%%%%%%%%%%%%%%%%%%%%%%%%%%%%%%%%%%%%%%%%%%%%%%%%%%
\subsection{Characteristic polynomial}

Since eigenvalues of the matrix $\Adj$ are the roots of its characteristic polynomial $p_\Adj(x)$, Theorem~\ref{thm:eigenvalue_is_root} establishes that the characteristic polynomial of any block tridiagonal matrix $\Adj$ of the form~\eqref{eq:block_tridiagonal_form} and the polynomial $\det\Pm_{L}(x)$ share the same roots. Here, we establish a yet stronger result: not only do these polynomials have the same roots, but the roots have the same multiplicities.

Recall that the characteristic polynomial of a matrix $\Adj$ that has $M$ distinct eigenvalues $\lambda_0,\ldots,\lambda_{M-1}$ with multiplicities $a_0,\ldots,a_{M-1}$ is defined as
\begin{equation}
\label{eq:characteristic_polynomial}
p_\Adj(x) = (x-\lambda_0)^{a_0}(x-\lambda_1)^{a_1}\ldots (x-\lambda_{M-1})^{a_{M-1}} = \prod_{m=0}^{M-1}(x-\lambda_m)^{a_m}.
\end{equation}
Since $a_0 + \ldots + a_{M-1}=N$, the degree of the characteristic polynomial is
$$
\deg p_\Adj(x) = N.
$$

Recall that a square matrix is \emph{diagonalizable} if it has a complete set of $N$ linearly independent eigenvectors, i.e., an eigenvalue $\lambda_m$ with multiplicity $a_m$ has exactly $a_m$ linearly independent eigenvectors~\cite{Lancaster:85,Gantmacher:59}. In this case, the $N\times N$ diagonalizable matrix $\Adj$ can be written as
\begin{equation}
\label{eq:eigendecomposition}
\Adj=\Vm \Eig \Vm^{-1}.
\end{equation}
Here, $\Vm\in\C^{N\times N}$ is the eigenvector matrix with columns corresponding to the eigenvectors of $\Adj$. The columns are arranged so that the first $a_0$ columns are the linearly independent eigenvectors corresponding to the eigenvalues $\lambda_0$, the next $a_1$ columns are eigenvectors corresponding to $\lambda_1$, etc. The matrix $\Eig\in\C^{N\times N}$ is the diagonal matrix of eigenvalues, arranged in the same order as the columns of $\Vm$. We write the eigenvector matrix as a block diagonal matrix
\begin{equation}
\label{eq:eigenvalue_matrix}
\Eig = \begin{pmatrix}
\lambda_0 \Id_{a_0} \\
& \lambda_1 \Id_{a_1} \\
&& \ddots \\
&&& \lambda_{M-1} \Id_{a_{M-1}}
\end{pmatrix}
\end{equation}

The following theorem relates the characteristic polynomials of diagonalizable matrices~\eqref{eq:block_tridiagonal_form} and determinants of corresponding matrix polynomials.
\begin{theorem}
\label{thm:char_is_det_diagonalizable}
If a block tridiagonal matrix $\Adj$ of the form~\eqref{eq:block_tridiagonal_form} is diagonalizable, its characteristic polynomial is equal, up to a normalization constant $c$, to the determinant of the corresponding matrix polynomial $\Pm_L(x)$:
\begin{equation}
\label{char_is_det}
p_\Adj(x) = \frac{1}{c}\det \Pm_{L}(x).
\end{equation}
\end{theorem}
\begin{proof}
The relation~\eqref{eq:eigenvector_via_P} in Theorem~\ref{thm:eigenvalue_is_root} establishes a bijective linear mapping between the eigenvectors of $\Adj$ corresponding to the eigenvalue $\lambda_m$ and the null-space of the scalar matrix $\Pm_{L}(\lambda_m)$, i.e., $\ker\Pm_{L}(\lambda_m)$ .
Since matrix $\Adj$ is diagonalizable, each eigenvalue $\lambda_m$ has exactly $a_m$ linearly independent eigenvectors.
Hence, by~\eqref{eq:eigenvector_via_P}, the null-space of $\Pm_{L}(\lambda_m)$ contains $a_m$ linearly independent vectors, and its dimension is
$$
\dim\ker\Pm_{L}(\lambda_m)=a_m.
$$

Furthermore, according to Theorem~\ref{thm:eigenvalue_is_root}, polynomials $p_\Adj(x)$ and $\det \Pm_{L}(x)$ have the same roots. Hence, similarly to~\eqref{eq:characteristic_polynomial}, we write
\begin{equation}
\label{eq:det_theorem}
\det \Pm_{L}(x) = c \prod_{m=0}^{M-1}(x-\lambda_m)^{b_m},
\end{equation}
where $b_m$ is the multiplicity of $\lambda_m$ as a zero of $\Pm_{L}(x)$, and $c$ is a normalization constant.
According to~\eqref{eq:degree_determinant}, the degree of the polynomial~\eqref{eq:det_theorem} is
\begin{eqnarray*}
\deg \det \Pm_{L}(x) &=& KL \\
&=& N.
\end{eqnarray*}
Hence, we obtain
\begin{equation}
\label{eq:sum_b}
b_0+b_1+\ldots+b_{M-1} = \deg \det \Pm_{L}(x) = N,
\end{equation}
and the polynomials $\det \Pm_{L}(x)$ and $p_\Adj(x)$ have the same degree.

The multiplicity $b_m$ is bounded from below by the dimension of the null-space of $\Pm_{L}(\lambda_m)$ (see Lemma 2.2 in~\cite{Duran:96}):
$$
b_m \geq \dim\ker\Pm_{L}(\lambda_m)=a_m.
$$
Hence, the polynomials $\det \Pm_{L}(x)$ and $p_\Adj(x)$ can have the same degree $N$ if and only if $b_m=a_m$ for all $0\leq m < M$. It follows that the polynomials~\eqref{eq:characteristic_polynomial} and~\eqref{eq:det_theorem} are equal, up to a normalization factor, i.e., that the equality~\eqref{char_is_det} holds.
\end{proof}

The following result follows immediately from the proof of Theorem~\ref{thm:char_is_det_diagonalizable}.

\begin{corollary}
\label{cor:number_of_eigenvalues}
An arbitrary (not necessarily diagonalizable) block tridiagonal matrix $\Adj$ of the form~\eqref{eq:block_tridiagonal_form} has at least
\begin{equation}
\label{eq:number_of_eigenvalues}
M \geq N/K
\end{equation}
distinct eigenvalues $\lambda_0$, $\ldots$, $\lambda_{M-1}$.
\end{corollary}
\begin{proof}
Since the multiplicity $a_m$ of an eigenvalue $\lambda_m$ corresponds to the dimension of the null-space of a $K\times K$ matrix $\Pm_{L}(\lambda_m)$, it satisfies $a_m \leq K$.
Hence, $N=a_0+\ldots+a_{M-1} \leq K + K +\ldots + K = KM$, which immediately yields~\eqref{eq:number_of_eigenvalues}.
\end{proof}

%%%%%%%%%%%%%%%%%%%%%%%%%%%%%%%%%%%%%%%%%%%%%%%%%%%%%%%%%%%%%%%%%
\subsection{Eigenvector matrix}

Theorems~\ref{thm:eigenvalue_is_root} and~\ref{thm:char_is_det_diagonalizable} demonstrate the connection between the eigenvalues of block tridiagonal matrices
and the zeros of corresponding matrix polynomials, as well as the relation between their eigenvectors and the null-space of the matrix polynomials.

In the following theorems, we provide closed-form expressions for the eigenvector matrices of diagonalizable block tridiagonal matrices~\eqref{eq:block_tridiagonal_form} and their inverses.

\begin{theorem}
\label{thm:eigenvector_matrix}
Consider a block tridiagonal matrix~\eqref{eq:block_tridiagonal_form} that is diagonalizable, i.e., can be represented in the form~\eqref{eq:eigendecomposition}.
As before, let $a_m$ denote the multiplicity of the eignevalue $\lambda_m$.

Let $\Hm_m$ denote a $K\times a_m$ matrix with columns given by basis vectors for the null-space of the scalar matrix $\Pm_{L}(\lambda_m)$.
Then the matrix
\begin{equation}
\label{eq:eigenvector_matrix}
\Vm = \begin{pmatrix}
\Pm_0(\lambda_0)\Hm_0 & \ldots & \Pm_0(\lambda_{M-1})\Hm_{M-1} \\
\vdots &  & \vdots \\
\Pm_{L-1}(\lambda_0)\Hm_0 & \ldots & \Pm_{L-1}(\lambda_{M-1})\Hm_{M-1}
\end{pmatrix}
\end{equation}
is an eigenvector matrix of $\Adj$, where the $(n,m)$th block is given by $\Pm_n(\lambda_m)\Hm_m$, $0\leq n < L$, $0\leq m < M$. In particular, the columns of the matrix
$$
\begin{pmatrix}
\Pm_0(\lambda_m)\Hm_m \\
\vdots \\
\Pm_{L-1}(\lambda_m)\Hm_m
\end{pmatrix}
$$
are linearly independent eigenvectors corresponding to the eigenvalue $\lambda_m$.
\end{theorem}
\begin{proof}
Using the block tridiagonal structure~\eqref{eq:block_tridiagonal_form} of $\Adj$, we can write the three-term relation~\eqref{eq:three_term_relation} for $0\leq n < N$ in the matrix-vector form as
\begin{equation}
\label{eq:AP}
\Adj \begin{pmatrix} \Pm_0(x) \\ \vdots \\ \Pm_{L-2}(x) \\ \Pm_{L-1}(x) \end{pmatrix} =
\begin{pmatrix} x\Pm_0(x) \\ \vdots \\ x\Pm_{L-2}(x) \\ x\Pm_{L-1}(x) - \Dm_{L-1}\Pm_{L}(x) \end{pmatrix}.
\end{equation}
Since the columns of $\Hm_m$ are basis vectors of $\ker\Pm_{L}(\lambda_m)$, the product $\Pm_{L}(\lambda_m)\Hm_m = \ZeroMatrix$ is a $K\times a_m$ zero matrix. Then it follows from~\eqref{eq:AP} that
\begin{eqnarray*}
&& \Adj \begin{pmatrix} \Pm_0(\lambda_m)\Hm_m \\ \vdots \\ \Pm_{L-2}(\lambda_m)\Hm_m \\ \Pm_{L-1}(\lambda_m)\Hm_m \end{pmatrix} \\
&=&
\begin{pmatrix} \lambda_m\Pm_0(\lambda_m)\Hm_m \\ \vdots \\ \lambda_m\Pm_{L-2}(\lambda_m)\Hm_m \\ \lambda_m\Pm_{L-1}(\lambda_m)\Hm_m-\Dm_{L-1}\Pm_{L}(\lambda_m)\Hm_m \end{pmatrix} \\
&=&
\begin{pmatrix} \Pm_0(\lambda_m)\Hm_m \\ \vdots \\ \Pm_{L-2}(\lambda_m)\Hm_m \\ \Pm_{L-1}(\lambda_m)\Hm_m \end{pmatrix}
\lambda_m\Id_{a_m}.
\end{eqnarray*}
Hence, for $\Vm$ in~\eqref{eq:eigenvector_matrix}, we obtain
$$
\Adj\Vm = \Vm \begin{pmatrix} \lambda_0\Id_{a_0} \\ & \ddots \\ && \lambda_{M-1}\Id_{a_{M-1}} \end{pmatrix} = \Vm\Eig.
$$
This decomposition is equal to the eigendecomposition~\eqref{eq:eigendecomposition}, which confirms that $\Vm$ is an eigenvector matrix for $\Adj$.
\end{proof}

Theorem~\ref{thm:eigenvector_matrix} gives a closed-form expression for the eigenvector matrix $\Vm$. When $\Vm$ is orthogonal, e.g., when the matrix $\Adj$ is symmetric or Hermitian, the expression~\eqref{eq:eigenvector_matrix} is also sufficient to find the inverse $\Vm^{-1}=\Vm^H$ of the eigenvector matrix. For the cases when $\Vm$ is not orthogonal, the following theorem provides a closed-form expression for its inverse.

\begin{theorem}
\label{thm:inverse_eigenvector_matrix}
Consider matrix polynomials $\widetilde{\Pm}_n(x)$ generated by the recursion
\begin{equation}
\label{eq:tilde_P}
x\cdot \widetilde{\Pm}_n(x) = \Dm^T_{n-1} \widetilde{\Pm}_{n-1}(x) + \Bm^T_n \widetilde{\Pm}_n(x) + \Cm^T_n \widetilde{\Pm}_{n+1}(x),
\end{equation}
with the initial conditions $\widetilde{\Pm}_{-1}(x)=\ZeroMatrix_K$ and $\widetilde{\Pm}_1(x)=\Id_K$. Assuming that all blocks $\Cm_n$ are non-singular,
i.e., $\det \Cm_n \neq 0$, the inverse of the eigenvector matrix~\eqref{eq:eigenvector_matrix} has the form
\begin{equation}
\label{eq:eigenvector_matrix_inverse}
\Vm^{-1} = \begin{pmatrix}
\widetilde{\Pm}_0(\lambda_0)\widetilde{\Hm}_0 & \ldots & \widetilde{\Pm}_0(\lambda_{M-1})\widetilde{\Hm}_{M-1} \\
\vdots &  & \vdots \\
\widetilde{\Pm}_{L-1}(\lambda_0)\widetilde{\Hm}_0 & \ldots & \widetilde{\Pm}_{L-1}(\lambda_{M-1})\widetilde{\Hm}_{M-1}
\end{pmatrix}^T,
\end{equation}
where each $\widetilde{\Hm}_m$ is a $K\times a_m$ matrix with columns given by the basis vectors for the null-space of the scalar matrix $\widetilde{\Pm}_{L}(\lambda_m)$.
\end{theorem}
\begin{proof}
This result follows directly from Theorem~\eqref{thm:eigenvector_matrix}. By transposing both sides of the eigendecomposition~\eqref{eq:eigendecomposition}, we obtain
$$
\Adj^T = \left(\Vm^{-1}\right)^T\Eig\Vm^T.
$$
Hence, $\left(\Vm^{-1}\right)^T$ is the eigenvector matrix of $\Adj^T$. But the transposed matrix $\Adj^T$ still has the block tridiagonal form~\eqref{eq:block_tridiagonal_form}.
Provided that the blocks $\Cm_n$ are non-singular, we recursively construct the corresponding matrix polynomials~\eqref{eq:tilde_P} and apply Theorem~\eqref{thm:eigenvector_matrix} to obtain
the eigenvector matrix for $\Adj^T$:
$$
\left(\Vm^{-1}\right)^T = \begin{pmatrix}
\widetilde{\Pm}_0(\lambda_0)\widetilde{\Hm}_0 & \ldots & \widetilde{\Pm}_0(\lambda_{M-1})\widetilde{\Hm}_{M-1} \\
\vdots &  & \vdots \\
\widetilde{\Pm}_{L-1}(\lambda_0)\widetilde{\Hm}_0 & \ldots & \widetilde{\Pm}_{L-1}(\lambda_{M-1})\widetilde{\Hm}_{M-1}
\end{pmatrix},
$$
which immediately yields the expression~\eqref{eq:eigenvector_matrix_inverse}.
\end{proof}

%%%%%%%%%%%%%%%%%%%%%%%%%%%%%%%%%%%%%%%%%%%%%%%%%%%%%%%%%%%%%%%%%
%%%%%%%%%%%%%%%%%%%%%%%%%%%%%%%%%%%%%%%%%%%%%%%%%%%%%%%%%%%%%%%%%
%%%%%%%%%%%%%%%%%%%%%%%%%%%%%%%%%%%%%%%%%%%%%%%%%%%%%%%%%%%%%%%%%

\section{Jordan decomposition of block tridiagonal matrices}
\label{sec:Jordan}

Theorem~\ref{thm:eigenvalue_is_root} shows the closed-form expression~\eqref{eq:eigenvector_via_P} for eigenvectors of a block tridiagonal matrix~\eqref{eq:block_tridiagonal_form} regardless of whether the matrix is diagonalizable or not. However, the results in Theorems~\ref{thm:char_is_det_diagonalizable},~\ref{thm:eigenvector_matrix}, and~\ref{thm:inverse_eigenvector_matrix} apply to diagonalizable matrices only. If the block tridiagonal matrix~\eqref{eq:block_tridiagonal_form} does not have a complete set of $N$ linearly independent eigenvectors, we cannot construct its eigendcomposition~\eqref{eq:eigendecomposition}. Instead, we need to find the generalized eigenvectors of the matrix and construct its Jordan decomposition~\cite{Lancaster:85,Gantmacher:59}.

Consider an eigenvector $\coord{v}$ that corresponds to an eigenvalue $\lambda$ of matrix $\Adj$. This eigenvector generates a \emph{Jordan chain} of \emph{generalized eigenvectors} $\coord{v}_0, \coord{v}_1, \ldots, \coord{v}_{R-1}$, where $\coord{v}_0=\coord{v}$, if these vectors satisfy the relation
\begin{equation}
\label{eq:generalized_eigenvector_relation1}
\Adj\coord{v}_r = \lambda\coord{v}_r + \coord{v}_{r-1}
\end{equation}
for $0\leq r < R$. The relation~\eqref{eq:generalized_eigenvector_relation1} for generalized eigenvectors is equivalent to the condition
\begin{equation}
\label{eq:generalized_eigenvector_relation2}
(\Adj - \lambda\Id_N)^{r+1}\coord{v}_r = 0.
\end{equation}

Determining which eigenvectors give rise to a Jordan chain is a challenge. Furthermore, similarly to  the calculation of eigenvectors, the direct calculation of generalized eigenvectors requires the solution of linear systems of the form~\eqref{eq:generalized_eigenvector_relation1}. Both problems are a prohibitively expensive procedures for matrices $\Adj$ of very large sizes.

In the following theorems, we propose two approaches that simplify the discovery and calculation of generalized eigenvectors of block tridiagonal matrices~\eqref{eq:block_tridiagonal_form}. Under appropriate conditions, we can take advantage of the matrix polynomials generated by the recurrence~\eqref{eq:recurrence} to determine whether an eigenvector generates a Jordan chain and then construct the corresponding generalized eigenvectors.

\begin{theorem}
\label{thm:generalized_eigenvectors1}
Consider an eigenvector $\coord{v}$ of the form~\eqref{eq:eigenvector_via_P} that corresponds to the eigenvalue $\lambda$ of a block tridiagonal matrix $\Adj$. This eigenvector generates a Jordan chain of $R$ generalized eigenvectors $\coord{v}_0, \coord{v}_1, \ldots, \coord{v}_{R-1}$ if the vector $\coord{u}$ from the null-space of $\Pm_L(\lambda)$ satisfies the property
\begin{equation}
\label{eq:condition_on_u}
\Pm_L^{(r)}(\lambda) \coord{u} = 0
\end{equation}
for $0 \leq r < R$, where
$$
\Pm_L^{(r)}(x) = \frac{d^r}{dx^r} \Pm_L(x)
$$
denotes the $r$th derivative of the matrix polynomial $\Pm_L(x)$.
The condition~\eqref{eq:condition_on_u} is equivalent to the requirement that the vector $\coord{u}$ belongs to the null-spaces of matrices $\Pm_L^{(r)}(\lambda)$ for $0 \leq r < R$.

The $r$th generalized eigenvector is then given by
\begin{equation}
\label{eq:vr_via_Pr}
\coord{v}_r = \frac{1}{r!} \begin{pmatrix} \Pm^{(r)}_0(\lambda) \\ \vdots \\ \Pm^{(r)}_{L-2}(\lambda) \\ \Pm^{(r)}_{L-1}(\lambda) \end{pmatrix} \coord{u}.
\end{equation}
\end{theorem}
\begin{proof}
Rewrite the equation~\eqref{eq:AP} in the proof of Theorem~\ref{thm:eigenvector_matrix} as
\begin{equation}
\label{eq:AP_is_xP}
\Adj \begin{pmatrix} \Pm_0(x) \\ \vdots \\ \Pm_{L-2}(x) \\ \Pm_{L-1}(x) \end{pmatrix} =
x \begin{pmatrix} \Pm_0(x) \\ \vdots \\ \Pm_{L-2}(x) \\ \Pm_{L-1}(x) \end{pmatrix} -
\begin{pmatrix} \ZeroMatrix_K \\ \vdots \\ \ZeroMatrix_K \\ \Dm_{L-1}\Pm_{L}(x) \end{pmatrix}.
\end{equation}

Taking the $r$th derivative of both sides of~\eqref{eq:AP_is_xP}, and using the fact that\footnote{
The equality~\eqref{eq:xPr} follows from the well-known property of the derivatives of function products:
\begin{equation*}
\Big( f(x) g(x) \Big)^{(r)} = \sum_{k=0}^r {r \choose k} f(x)^{(k)} g(x)^{(r-k)}.
\end{equation*}
Setting $f(x)=x$ and using the fact that $x^{(1)}=1$ and $x^{(r)}=0$ for $r>1$, we immediately obtain~\eqref{eq:xPr}.}
\begin{equation}
\label{eq:xPr}
\left( x \begin{pmatrix} \Pm_0(x) \\ \vdots \\ \Pm_{L-2}(x) \\ \Pm_{L-1}(x) \end{pmatrix} \right)^{(r)} =
x \begin{pmatrix} \Pm^{(r)}_0(x) \\ \vdots \\ \Pm^{(r)}_{L-2}(x) \\ \Pm^{(r)}_{L-1}(x) \end{pmatrix}  +
r \begin{pmatrix} \Pm^{(r-1)}_0(x) \\ \vdots \\ \Pm^{(r-1)}_{L-2}(x) \\ \Pm^{(r-1)}_{L-1}(x) \end{pmatrix},
\end{equation}
we obtain
\begin{equation}
\label{eq:APr_is_xPr}
\Adj \begin{pmatrix} \Pm^{(r)}_0(x) \\ \vdots \\ \Pm^{(r)}_{L-2}(x) \\ \Pm^{(r)}_{L-1}(x) \end{pmatrix} =
x \begin{pmatrix} \Pm^{(r)}_0(x) \\ \vdots \\ \Pm^{(r)}_{L-2}(x) \\ \Pm^{(r)}_{L-1}(x) \end{pmatrix} +
r \begin{pmatrix} \Pm^{(r-1)}_0(x) \\ \vdots \\ \Pm^{(r-1)}_{L-2}(x) \\ \Pm^{(r-1)}_{L-1}(x) \end{pmatrix} -
\begin{pmatrix} \ZeroMatrix_K \\ \vdots \\ \ZeroMatrix_K \\ \Dm_{L-1}\Pm_{L}^{(r)}(x) \end{pmatrix}.
\end{equation}

The closed-form expression~\eqref{eq:vr_via_Pr} is proven by induction. The base case for $r=0$ is established by~\eqref{eq:eigenvector_via_P} in Theorem~\eqref{thm:eigenvalue_is_root}. Assuming that~\eqref{eq:vr_via_Pr} holds for $\coord{v}_{r-1}$, we substitute $x=\lambda$ into~\eqref{eq:APr_is_xPr} and multiply both sides of~\eqref{eq:APr_is_xPr} by vector $\coord{u}$
to obtain the equality
\begin{equation}
\label{eq:inductive_step}
\Adj \begin{pmatrix} \Pm^{(r)}_0(\lambda) \\ \vdots \\ \Pm^{(r)}_{L-2}(\lambda) \\ \Pm^{(r)}_{L-1}(\lambda) \end{pmatrix}\coord{u} =
\lambda \begin{pmatrix} \Pm^{(r)}_0(\lambda) \\ \vdots \\ \Pm^{(\lambda)}_{L-2}(x) \\ \Pm^{(r)}_{L-1}(\lambda) \end{pmatrix}\coord{u} +
r \big( (r-1)! \coord{v}_{r-1} \big) -
\begin{pmatrix} \ZeroMatrix_K \\ \vdots \\ \ZeroMatrix_K \\ \Dm_{L-1}\Pm_{L}^{(r)}(\lambda)\coord{u} \end{pmatrix}.
\end{equation}
Recall that according to~\eqref{eq:condition_on_u} $\Pm_{L}^{(r)}(\lambda)\coord{u}=0$. By comparing the equation~\eqref{eq:inductive_step} with the definition~\eqref{eq:generalized_eigenvector_relation1} of generalized eigenvectors, we conclude
that~\eqref{eq:APr_is_xPr} holds for $\coord{v}_r$. Hence, the proof by induction is complete.
\end{proof}

Theorem~\eqref{thm:generalized_eigenvectors1} provides a method to check whether an eigenvector $\coord{v}$ with a corresponding eigenvalue $\lambda$ generates a Jordan chain of generalized eigenvectors. If the condition~\eqref{eq:condition_on_u} does not hold, the following theorem provides a significantly more general approach to the determination of existence and the subsequent calculation of generalized eigenvectors.

\begin{theorem}
\label{thm:generalized_eigenvectors2}
Consider an eigenvector $\coord{v}$ of the form~\eqref{eq:eigenvector_via_P} that corresponds to the eigenvalue $\lambda$ of a block tridiagonal matrix $\Adj$. This eigenvector generates a Jordan chain of $R$ generalized eigenvectors $\coord{v}_0=\coord{v}, \coord{v}_1, \ldots, \coord{v}_{R-1}$ if for each $0\leq r < R$, the scalar $(-1)^r\lambda^{r+1}$ is an eigenvalue of the matrix $(\Adj-\lambda\Id_N)^{r+1} + (-1)^r\lambda^{r+1}\Id_N$. In this case, each generalized eigenvector $\coord{v}_r$ is the corresponding eigenvector, i.e., it satisfies
\begin{equation}
\label{eq:condition_on_lambda}
\Big( (\Adj-\lambda\Id_N)^{r+1} + (-1)^r\lambda^{r+1}\Id_N  \Big)\coord{v}_r = (-1)^r\lambda^{r+1} \coord{v}_r.
\end{equation}
\end{theorem}
\begin{proof}
This result follows directly from the definition~\eqref{eq:generalized_eigenvector_relation2} of generalized eigenvectors.
\end{proof}

The significance of Theorem~\ref{thm:generalized_eigenvectors2} lies in the fact that powers of block tridiagonal matrices are themselves block tridiagonal matrices with blocks of larger sizes. Namely, the matrix $\Adj-\lambda\Id_N$ is a block tridiagonal matrix with the structure~\eqref{eq:block_tridiagonal_form}. Taken to the power $r$, the matrix $\Adj_r=(\Adj-\lambda\Id_N)^{r+1} + (-1)^r\lambda^{r+1}\Id_N$ in~\eqref{eq:condition_on_lambda} is also a block tridiagonal matrix of the form~\eqref{eq:block_tridiagonal_form}, but with blocks $\Bm_n$, $\Cm_n$, and $\Dm_n$ having sizes $rK$. Hence, we can use the recurrence~\eqref{eq:recurrence} to construct $rK\times rK$ matrix polynomials that correspond to matrix $\Adj_r$. Then, given the eigenvalue $\lambda$, we can use Theorem~\ref{thm:eigenvalue_is_root} to check whether $(-1)^r\lambda^{r+1}$ is an eigenvalue of $\Adj_r$ and construct the corresponding eigenvector $\coord{v}_r$.

%%%%%%%%%%%%%%%%%%%%%%%%%%%%%%%%%%%%%%%%%%%%%%%%%%%%%%%%%%%%%%%%%
%%%%%%%%%%%%%%%%%%%%%%%%%%%%%%%%%%%%%%%%%%%%%%%%%%%%%%%%%%%%%%%%%
%%%%%%%%%%%%%%%%%%%%%%%%%%%%%%%%%%%%%%%%%%%%%%%%%%%%%%%%%%%%%%%%%

\section{Discussion}

In this chapter, we discuss how the results obtained in Chapters~\ref{sec:Eigendecomposition} and~\ref{sec:Jordan} can be advantageous to the calculation of eigenvalues and eigenvectors of block tridiagonal matrices.

%%%%%%%%%%%%%%%%%%%%%%%%%%%%%%%%%%%%%%%%%%%%%%%%%%%%%%%%%%%%%%%%%
\subsection{Eigenvalue calculation}

The relationship between the eigenvalues of a block tridiagonal matrix $\Adj$ and the roots of the determinant of the corresponding matrix polynomial $\Pm_L(x)$ in Theorem~\ref{thm:eigenvalue_is_root} provides an alternative way of calculating the eigenvalues of $\Adj$. While the determinant $\det \Pm_L(x)$ still needs to be factored in order to find the eigenvalues, this problem can be simplified when the blocks $\Bm_n$, $\Cm_n$, and $\Dm_n$ in~\eqref{eq:block_tridiagonal_form} have additional structural properties.

As an illustration, consider the case when all blocks of matrix $\Adj$ commute with each other. If at least one of the blocks is diagonalizable, then all of them are diagonzaliable and have the same eigenvectors, since these are matrices over the complex numbers~\cite{Lancaster:85,Gantmacher:59}. In this case the blocks are factored as
\begin{eqnarray*}
\Bm_n &=& \Um \Eig_{\Bm_n} \Um^{-1}, \\
\Cm_n &=& \Um \Eig_{\Cm_n} \Um^{-1}, \\
\Dm_n &=& \Um \Eig_{\Dm_n} \Um^{-1}.
\end{eqnarray*}
Here, the matrices $\Eig_{\Bm_n}$, $\Eig_{\Cm_n}$, and $\Eig_{\Dm_n}$ are diagonal eigenvalue matrices for the corresponding blocks, and $\Um$ is their common eigenvector matrix.

Hence, the matrix $\Adj$ is similar to the block tridiagonal matrix
\begin{eqnarray*}
&&
\begin{pmatrix}
\Eig_{\Bm_0} & \Eig_{\Dm_0} \\
\Eig_{\Cm_0} & \Eig_{\Bm_1} & \ddots \\
& \ddots & \ddots & \Eig_{\Dm_{L-2}} \\
&& \Eig_{\Cm_{L-2}} & \Eig_{\Bm_{L-1}}
\end{pmatrix} \\
&=&
\begin{pmatrix}
\Um^{-1} \\
& \Um^{-1} \\
&& \ddots \\
&&& \Um^{-1}
\end{pmatrix}
\Adj
\begin{pmatrix}
\Um \\
& \Um \\
&& \ddots \\
&&& \Um
\end{pmatrix}.
\end{eqnarray*}

Since similar matrices have the same eigenvalues and characteristic polynomials~\cite{Lancaster:85,Gantmacher:59}, it suffices to construct the characteristic polynomial of the block tridiagonal matrix with diagonal blocks $\Eig_{\Bm_n}$, $\Eig_{\Cm_n}$, and $\Eig_{\Dm_n}$. The recurrence~\eqref{eq:recurrence} for the corresponding matrix polynomials $\Qm_n(x)$ becomes
\begin{equation}
\label{eq:recurrence_diagonal}
\Qm_{n+1}(x) = \Eig_{\Dm_n}^{-1} \left( x\cdot \Qm_n(x) - \Eig_{\Bm_n} \Qm_n(x) - \Eig_{\Cm_{n-1}} \Qm_{n-1}(x) \right).
\end{equation}
Since $\Qm_{-1}(x) = \ZeroMatrix_K$ and $\Qm_0(x)=\Id_K$, each polynomial $\Qm_n(x)$ generated by the recurrence~\eqref{eq:recurrence_diagonal} is a diagonal matrix. The elements on its diagonal are polynomials of degree $n$. Hence, the determinant of $\Qm_L(x)$ is a product of $K$ polynomials of degree $L$.

The representation of the characteristic polynomial by a product of $K$ polynomials of degree $L$ yields substantial reduction in the cost of eigenvalue calculation. The factorization of the polynomial $p_\Adj(x)$ of degree $N$, in general, requires $O(N^3)$ operations. In comparison, the factorization of $K$ polynomials of degree $L$ requires $O(KL^3)=O(N^3/K^2)$ operations. Hence, the eigenvalue calculation is accelerated by a factor of $K^2$.

%%%%%%%%%%%%%%%%%%%%%%%%%%%%%%%%%%%%%%%%%%%%%%%%%%%%%%%%%%%%%%%%%
\subsection{Eigenvector calculation}

The closed-form expression~\eqref{eq:eigenvector_matrix} for the eigenvector matrix yields a substantial reduction of the computation cost for any block tridiagonal matrix~\eqref{eq:block_tridiagonal_form}. In general, the direct computation of eigenvectors requires solving $N$ equations with $N$ unknowns, with a total cost of $O(N^3)$ operations.
Instead, we can calculate the bases of the null-spaces of $M$ matrices $\Pm_{L}(\lambda_m)$, which requires only $O(K^2M)$ operations, and compute $N$ products of $K\times K$ matrices with vectors of length $K$, which requires $O(K^2N)=O(N^3/L^2)$ operations. The total operations required are $O(K^2M)+O(N^3/L^2) = O(N^3/L^2)$. Hence, the eigenvector calculation is accelerated by a factor of $L^2$.

%%%%%%%%%%%%%%%%%%%%%%%%%%%%%%%%%%%%%%%%%%%%%%%%%%%%%%%%%%%%%%%%%
%%%%%%%%%%%%%%%%%%%%%%%%%%%%%%%%%%%%%%%%%%%%%%%%%%%%%%%%%%%%%%%%%
%%%%%%%%%%%%%%%%%%%%%%%%%%%%%%%%%%%%%%%%%%%%%%%%%%%%%%%%%%%%%%%%%
\section{Spectral analysis of ``spider'' graphs}

In this chapter, we apply the theory presented in this paper to the spectral analysis of graphs, i.e., the computation of eigenvalues and eigenvectors of adjacency matrices of graphs. Spectral graph theory finds many important applications, including among others machine learning and data mining~\cite{Chapelle:06,Luxburg:07}, ranking algorithms~\cite{Brin:98}, and image processing~\cite{Shi:00}.

%%%%%%%%%%%%%%%%%%%%%%%%%%%%%%%%%%%%%%%%%%%%%%%%%%%%%%%%%%%%%%%%%
\subsection{``Spider'' graphs}

We consider the spectral analysis of ``spider'' graphs~\cite{Dette:06, Grunbaum:09}, since their adjacency matrices have the required block tridiagonal structure~\eqref{eq:block_tridiagonal_form}.
These graphs can be seen as generalizations of star graphs. A ``spider'' graph consists of $K$ legs with $L$ nodes on each leg, such that nodes on each leg are connected sequentially, and a few nodes on different legs can be connected to each other. These graphs arise in different problems and settings. For example, sampled pulses in magnetic resonance imaging and X-ray tomography form a ``spider'' graph in $K$-space~\cite{Lauterbur:73}. These graphs also can represent the layout of sensor networks or router interconnection topology.

% Figure: spider graphs
\begin{figure}
  \begin{center}
    \subfigure[]{\label{fig:spider1}\includegraphics[scale=0.5]{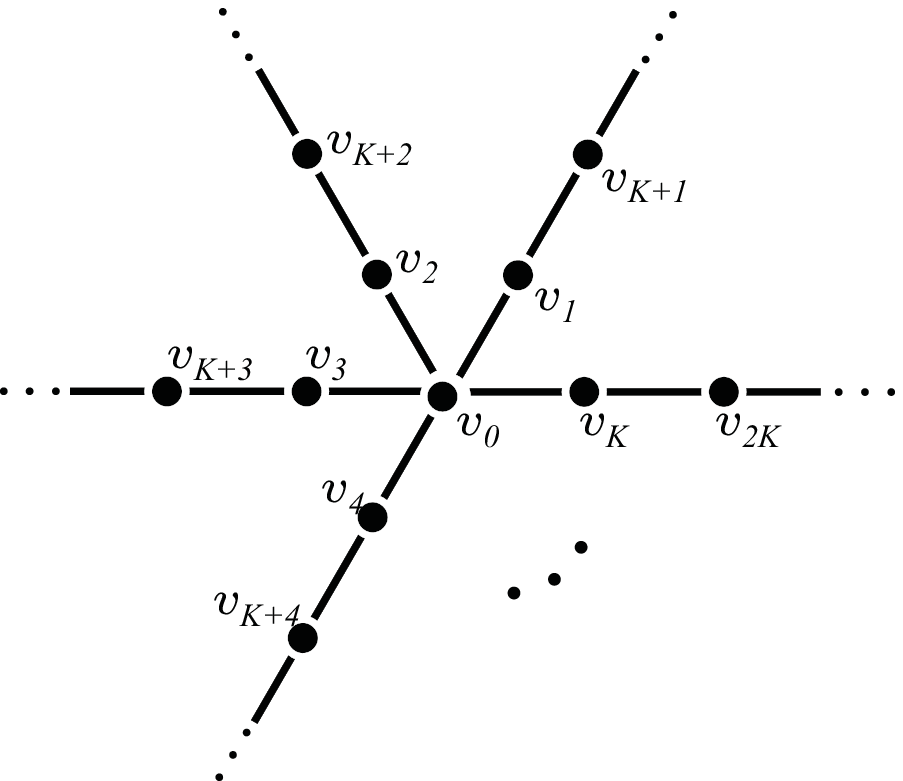}}
    \hspace{5mm}
    \subfigure[]{\label{fig:spider2}\includegraphics[scale=0.5]{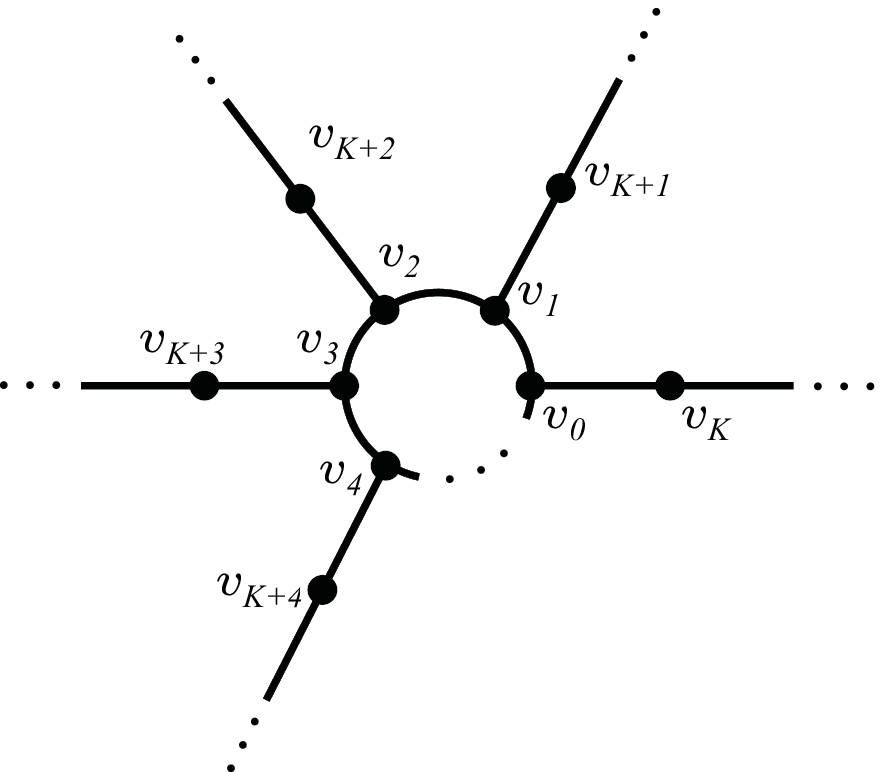}}
  \end{center}
  \vspace{-2mm}
\caption{\label{fig:spider_graphs} Examples of ``spider'' graphs with $K$ rays.}
\vspace{-5mm}
\end{figure}
% ----------------------------------------

Examples of ``spider'' graphs are shown in Fig.~\ref{fig:spider_graphs}. For simplicity of discussion, assume that the graphs in Fig.~\ref{fig:spider_graphs} are undirected and unweighted, i.e., all edges are undirected and have the same weight $1$. If we label the $\ell$th node on the $k$th leg of a ``spider'' graph as $v_{k+\ell K}$, the graph adjacency matrix becomes a $N\times N$ block tridiagonal matrix of the form
\begin{equation}
\label{eq:star_ajdacency_matrix}
\Adj = \begin{pmatrix}
\Bm & \Id_K \\
\Id_K & \ZeroMatrix_K & \ddots \\
& \ddots & \ddots & \Id_K \\
&& \Id_K & \ZeroMatrix_K
\end{pmatrix},
\end{equation}
where $N=KL$. Matrix $\Bm$ is a $K\times K$ matrix that captures the connection pattern between different legs in the middle of the graph. For the graph in Figs.~\ref{fig:spider1}, this matrix is
\begin{equation}
\label{eq:spider_B}
\Bm = \begin{pmatrix}
0 & 1 & \dots & 1 \\
1  \\
\vdots & & \ZeroMatrix_{K-1} \\
1
\end{pmatrix},
\end{equation}
and for the graph in Fig.~\ref{fig:spider2} this matrix is
$$
\Bm = \begin{pmatrix}
0 & 1 &&1 \\
1 & 0 & \ddots \\
& \ddots & \ddots& 1 \\
1&&1 & 0
\end{pmatrix}.
$$

The generating recurrence~\eqref{eq:recurrence} for the matrix polynomials corresponding to the block tridiagonal matrix~\eqref{eq:star_ajdacency_matrix} is
\begin{equation}
\label{eq:recurrence_spider}
\Pm_n(x) = x\Pm_{n-1}(x) - \Pm_{n-2}(x),
\end{equation}
with $\Pm_0(x) = \Id_K$ and $\Pm_1(x) = x\Id_K-\Bm$. It follows from~\eqref{eq:recurrence_spider} that the generated matrix polynomials have the form
\begin{equation}
\label{eq:P_spider}
\Pm_n(x) = U_n(x/2)\Id_K + U_{n-1}(x/2)\Bm,
\end{equation}
where $U_n(x)$ are Chebyshev polynomials of the second kind. Recall that Chebyshev polynomials of the second kind $U_n(x)$ satisfy the
three-term recurrence $U_{n}(x) = 2xU_{n-1}(x) - U_{n-2}(x)$ with initial conditions $U_0(x)=1$ and $U_1(x)=2x$~\cite{Mason:02}.
The $n$th Chebyshev polynomial $U_n(x)$ has exactly $n$ distinct, simple roots
\begin{equation}
\label{eq:cheb_root}
x_k=\frac{\cos(k+1)\pi}{2n+1}
\end{equation}
for $0\leq k < n$.

Next, we determine the characteristic polynomial of the adjacency matrix~\eqref{eq:star_ajdacency_matrix} and its eigenvector matrix, as well as identify a fast computation algorithm for its eigenbasis expansion.

%%%%%%%%%%%%%%%%%%%%%%%%%%%%%%%%%%%%%%%%%%%%%%%%%%%%%%%%%%%%%%%%%
\subsection{Eigenvalues of a ``spider'' graph}

As a running example for the rest of this chapter, we consider the graph in Fig.~\ref{fig:spider1} with the corresponding block $\Bm$ given by~\eqref{eq:spider_B}.
Since the adjacency matrix~\eqref{eq:star_ajdacency_matrix} is symmetric, it is diagonalizable. Hence, according to Theorem~\ref{thm:char_is_det_diagonalizable},
its eigenvalues multiplicities are the same as the roots of the determinant of
$$
\Pm_L(x) =
\begin{pmatrix}
U_{L}(x) & U_{L-1}(x) & \dots & U_{L-1}(x) \\
U_{L-1}(x) & U_{L}(x)  \\
\vdots && \ddots \\
U_{L-1}(x) &&& U_{L}(x)
\end{pmatrix},
$$
as given by~\eqref{eq:P_spider}. It is straightforward to demonstrate by induction that
\begin{eqnarray}
\label{eq:det_spider}
\nonumber
\det\Pm_L(x) &=&
U_L^{K-2}(x/2)\left( U_L^2(x/2) - (K-1)U_{L-1}^2(x/2) \right) \\
&=& U_L^{K-2}(x/2) \\
\nonumber
&& \times \left( U_L(x/2) + \sqrt{K-1}U_{L-1}(x/2) \right) \\
\nonumber
&& \times \left( U_L(x/2) - \sqrt{K-1}U_{L-1}(x/2) \right) .
\end{eqnarray}

The roots $\alpha_0,\ldots,\alpha_{L-1}$ of the polynomial $U_L(x/2)$ are given by~\eqref{eq:cheb_root}.
Hence, $\Adj$ has $L$ eigenvalues $\alpha_k=2\cos(k+1)\pi / (2L+1)$, $0\leq k < L$, each with multiplicity $K-2$.
To determine the remaining eigenvalues of $\Adj$, we need to find the roots
$\beta_0,\ldots,\beta_{L-1}$ of the polynomial $U_L(x/2) + \sqrt{K-1}U_{L-1}(x/2)$; and
$\gamma_0,\ldots,\gamma_{L-1}$ denote the roots of the polynomial $U_L(x/2) - \sqrt{K-1}U_{L-1}(x/2)$.
This can be done by factoring the corresponding polynomials directly. Alternatively, we can use the property that $\beta_k$ and $\gamma_k$
are the eigenvalues of matrices~\cite{Askey:87,Pueschel:05e,Sandryhaila:12}
$$
\frac{1}{2}
\begin{pmatrix}
& 1 \\
1& & \ddots \\
& \ddots && 1 \\
&& 1 & -\sqrt{K-1}
\end{pmatrix}
\text{     and     }
\frac{1}{2}
\begin{pmatrix}
& 1 \\
1& & \ddots \\
& \ddots && 1 \\
&& 1 & \sqrt{K-1}
\end{pmatrix},
$$
respectively. Furthermore, all $\beta_k$ and $\gamma_k$ are simple, distinct roots.
Hence, $\Adj$ has $2L$ eigenvalues $\beta_k$ and $\gamma_k$, $0\leq k < L$, each with multiplicity $1$.

%%%%%%%%%%%%%%%%%%%%%%%%%%%%%%%%%%%%%%%%%%%%%%%%%%%%%%%%%%%%%%%%%
\subsection{Eigenvectors of a ``spider'' graph}

To construct the eigenvector matrix $\Vm$ using Theorem~\ref{thm:eigenvector_matrix}, we determine bases of the null-spaces of the scalar matrices $\Pm_L(\alpha_k)$, $\Pm_L(\beta_k)$, and $\Pm_L(\gamma_k)$.

Since $U_{L-1}(\alpha_k/2)\neq 0$, a basis of the null-space of $\Pm_L(\alpha_k)=U_{L-1}(\alpha_k/2)\Bm$ is the same as the basis of the null-space of $\Bm$. It is given by columns of the $K\times (K-2)$ matrix
\begin{equation}
\label{eq:H}
\Hm = \begin{pmatrix}
0 & 0 & \dots & 0 \\
1  \\
-1 & 1   \\
 & -1 & \ddots   \\
 &  & \ddots & 1   \\
 &&& -1
\end{pmatrix}.
\end{equation}

For $\lambda\in\{\beta_0,\ldots,\beta_{L-1},\gamma_0,\ldots,\gamma_{L-1}\}$, a basis of the null-space of the scalar matrix $\Pm_L(\lambda) = U_L(\lambda/2)\Id_K + U_{L-1}(\lambda/2)\Bm$ is given by the vector
\begin{equation}
\label{eq:h_lambda}
\coord{h}_\lambda = \begin{pmatrix} U_L(\lambda/2) & U_{L-1}(\lambda/2) & \ldots & U_{L-1}(\lambda/2) \end{pmatrix}^T.
\end{equation}

Hence, the eigenvector matrix for the adjacency matrix~\eqref{eq:star_ajdacency_matrix} can be written as a block matrix
\begin{equation}
\label{eq:spider_eigenvector_matrix}
\Vm = \left( \Vm_\alpha\, \Vm_\beta\, \Vm_\gamma \right),
\end{equation}
where
\begin{eqnarray*}
\Vm_\alpha &=& \begin{pmatrix}
\Pm_0(\alpha_0)\Hm & \dots & \Pm_0(\alpha_{L-1})\Hm  \\
\vdots && \vdots  \\
\Pm_{L-1}(\alpha_0)\Hm & \dots & \Pm_{L-1}(\alpha_{L-1})\Hm
\end{pmatrix}, \\
 \\
\Vm_\beta &=& \begin{pmatrix}
\Pm_0(\beta_0)\coord{h}_{\beta_0} & \dots & \Pm_0(\beta_{L-1})\coord{h}_{\beta_{L-1}}\\
\vdots && \vdots \\
\Pm_{L-1}(\beta_0)\coord{h}_{\beta_0} & \dots & \Pm_{L-1}(\beta_{L-1})\coord{h}_{\beta_{L-1}}
\end{pmatrix},\\
 \\
\Vm_\gamma &=& \begin{pmatrix}
\Pm_0(\gamma_0)\coord{h}_{\gamma_0} & \dots & \Pm_0(\gamma_{L-1})\coord{h}_{\gamma_{L-1}} \\
\vdots && \vdots \\
\Pm_{R-1}(\gamma_0)\coord{h}_{\gamma_0} & \dots & \Pm_{L-1}(\gamma_{L-1})\coord{h}_{\gamma_{L-1}}
\end{pmatrix}.
\end{eqnarray*}

%%%%%%%%%%%%%%%%%%%%%%%%%%%%%%%%%%%%%%%%%%%%%%%%%%%%%%%%%%%%%%%%%
\subsection{Fast eigenvector expansion algorithm}

We can use the closed-form expression~\eqref{eq:spider_eigenvector_matrix} for the eigenvector matrix of the ``spider'' graph to construct a fast algorithm for the eigenvector expansion. The expansion of a vector $\coord{y}\in\C^N$ is given by the matrix-vector product
\begin{equation}
\label{eq:expansion}
\coord{\hat{y}} = \Vm^{-1} \coord{y}.
\end{equation}
In the case of symmetric $\Adj$, the expansion~\eqref{eq:expansion} simplifies to
\begin{equation}
\label{eq:expansion_symm}
\coord{\hat{y}} = \Vm^T \coord{y}.
\end{equation}
We construct a fast computation algorithm for the eigenvector expansion~\eqref{eq:expansion_symm} by decomposing the matrix $\Vm$ into smaller matrices that can be computed very fast, which is an effective approach for the construction of fast algorithms for important linear operators~\cite{Pueschel:03a,Pueschel:08c,Sandryhaila:11a}.

Since $\alpha_k$ are roots of the polynomial $U_L(x/2)$, we can use the relation~\eqref{eq:P_spider} between the matrix polynomials and Chebyshev polynomials to rewrite $\Vm_\alpha$
in terms of the discrete sine transform as~\cite{Pueschel:05e,Pueschel:08b}
\begin{equation}
\label{eq:Valpha}
\Vm_\alpha = \DSTI_L\otimes \Hm + \left(\Cm \DSTI_L\right) \otimes \left(\Bm\Hm\right),
\end{equation}
where
\begin{equation}
\label{eq:C}
\Cm = \begin{pmatrix}
0 & 1 \\
& 0 & \ddots \\
&& \ddots & 1 \\
&&& 0
\end{pmatrix}
\end{equation}
is a $L\times L$ matrix; $\DSTI_R$ is the unscaled discrete sine transform of the first type of size $L$~\cite{Pueschel:08b}; $H$ is given by~\eqref{eq:H}; and $\otimes$ denotes the tensor (Kronecker) product of matrices~\cite{Bernstein:09}.

Similarly, we use~\eqref{eq:P_spider} to rewrite both matrices $\Vm_\beta$ and $\Vm_\gamma$ as
\begin{equation}
\label{eq:Vlambda}
\Vm_\lambda = \Big( \Tm(\lambda)\otimes \Id_K + \Cm \Tm(\lambda)\otimes \Bm \Big) \widehat{\Hm}_\lambda,
\end{equation}
where the parameter $\lambda$ stands for $\beta$ or $\gamma$. The matrix $\Tm(\lambda)$ is a $L\times L$ matrix with $(n,k)$th element given by
\begin{equation}
\label{eq:T}
\Tm(\lambda)_{n,k} = U_n(\lambda_k/2)
\end{equation}
for $0\leq n,k<L$. The matrix $\Cm$ is given in~\eqref{eq:C}. The matrix
$$
\widehat{\Hm}_\lambda = \begin{pmatrix}
\coord{h}_{\lambda_0}\\
&\ddots\\
&&\coord{h}_{\lambda_{L-1}}
\end{pmatrix}
$$
is a $N\times L$ block diagonal matrix obtained by using vectors $\coord{h}_{\lambda_0},\ldots,\coord{h}_{\lambda_{L-1}}$, given by~\eqref{eq:h_lambda},
as elements of a $L\times L$ diagonal matrix.

In general, the calculation of the matrix-vector product~\eqref{eq:expansion} requires $O(N^2)$ operations. However, the decompositions~\eqref{eq:Valpha} and~\eqref{eq:Vlambda} yield a fast algorithm for the calculation of matrix-vector products of $\Vm_\alpha^T$, $\Vm_\beta^T$, and $\Vm_\gamma^T$ with the vector $\coord{y}$, and hence, for the product~\eqref{eq:expansion}, as explained next.

The calculation of $\DSTI_L$ in~\eqref{eq:Valpha} requires $O(L\log L)$ operations~\cite{Pueschel:03a,Pueschel:08c}. The calculations of $\Bm$ and $\Hm$ require $O(K)$ operations each,
and $\Cm$ does not require any operations. Hence, the calculation of a matrix-vector product with $\Vm_\alpha$ according to~\eqref{eq:Valpha} requires $O(L\log L)K+O(K)L = O(KL\log L) = O(N\log L)$ operations in total.

The matrix $\Tm(\gamma)$ in~\eqref{eq:T} is called a \emph{1-D nearest-neighbor transform}~\cite{Sandryhaila:12}. Its calculation requires $O(L\log^2 L)$ operations~\cite{Driscoll:97}.
The calculation of $\Bm$ requires $O(K)$ operations, the calculation of $\widehat{\Hm}_\lambda$ requires $O(N)$ operations, and $\Cm$ does not require any operations.
Hence, the calculations of matrix-vector products with $\Vm_\beta$ and $\Vm_\gamma$ using~\eqref{eq:Vlambda} each require $O(L\log^2 L)K + O(L\log^2 L)K + O(K)L + O(N) = O(N\log^2 L)$ operations.

Hence,~\eqref{eq:expansion} requires $O(N\log L) + O(N\log^2 L) + O(N\log^2 L) = O(N\log^2 L)$ operations altogether, instead of $O(N^2)$.
Thus, we have reduced the computational cost by the factor $N/\log^2 L$ and obtained a fast algorithm for the eigenvector expansion for the block tridiagonal matrix~\eqref{eq:star_ajdacency_matrix}.
The plot in Fig.~\ref{fig:acceleration} illustrates the improvement achieved by this algorithm. It shows the computational savings achieved on a ``spider'' graph with $K=8$ rays as the number of nodes $L$ on each ray increases.

Also, notice that in the degenerate case when $K=2$, the graph in Fig.~\ref{fig:spider1} reduces to an undirected line graph. The corresponding graph Fourier transform then is $\DSTI_N$ and requires only $O(N\log N)=O(N\log L)$ operations, instead of $O(N\log^2 L)$ operations~\cite{Pueschel:08c}.

% Figure: acceleration
\begin{figure}
  \vspace{-15mm}
  \begin{center}
    \includegraphics[scale=0.4]{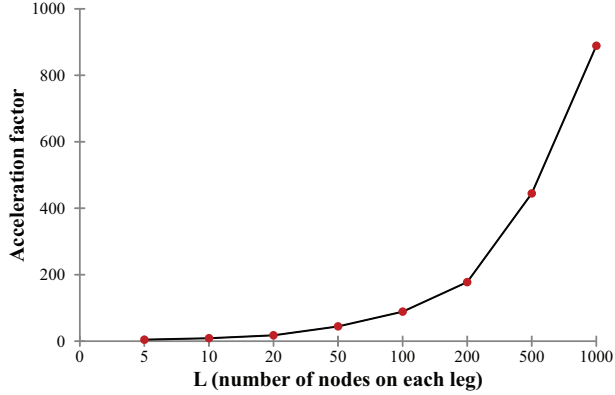}
  \end{center}
  \vspace{-2mm}
\caption{\label{fig:acceleration} Computation speed-up achieved by the fast algorithm for the eigenvector expansion for the adjacency matrix~\eqref{eq:star_ajdacency_matrix} of a ``spider'' graph with $K=8$ legs and $L$ of nodes on each leg. }
%\vspace{-5mm}
\end{figure}
% ----------------------------------------

The constructed fast algorithm uses a \emph{divide-and-conquer} approach similar to the Fast Fourier Transform algorithms~\cite{Duhamel:90} and their variations for discrete cosine and sine transforms~\cite{Pueschel:03a,Pueschel:08c}, real discrete Fourier transform~\cite{Voronenko:09}, Walsh-Hadamard transform~\cite{Johnson:00}, and other linear transforms~\cite{Sandryhaila:11a}. Divide-and-conquer algorithms compute their corresponding matrix-vector products by factoring the matrix into a product of sparse matrices that can then be factored further into products of yet sparser matrices. This approach reduces the computational cost from $O(N^2)$ operations to $O(N\log N)$ or $O(N\log^2 N)$ and makes it feasible to compute the transforms of very large sizes. Moreover, the structure of divide-and-conquer algorithms makes them particularly suitable for implementations on multi-core platforms~\cite{Franchetti:09} yielding further speed-ups.

%%%%%%%%%%%%%%%%%%%%%%%%%%%%%%%%%%%%%%%%%%%%%%%%%%%%%%%%%%%%%%%%%
%%%%%%%%%%%%%%%%%%%%%%%%%%%%%%%%%%%%%%%%%%%%%%%%%%%%%%%%%%%%%%%%%
%%%%%%%%%%%%%%%%%%%%%%%%%%%%%%%%%%%%%%%%%%%%%%%%%%%%%%%%%%%%%%%%%

\section{Conclusion}

This paper considers the problem of exact eigendecomposition of block tridiagonal matrices. We study the relation between this class of matrices and appropriately generated matrix polynomials. We connect eigenvalues of block tridiagonal matrices with the zeros of the matrix polynomials and relate matrix eigenvectors to the null-spaces of the matrix polynomials evaluated at the eigenvalues, which are scalar matrices of much smaller dimensions.

Our framework reduces the cost of the eigendecomposition of block tridiagonal matrices, since it replaces direct calculations of large matrices with equivalent problems of polynomial factorization and determination of null-spaces for significantly smaller matrices. Furthermore, it yields a closed-form expression for eigenvector matrices that can lead to the discovery of fast algorithms for the eigenvector expansion, as we illustrated with the example of ``spider'' graphs.

%\mypar{Future work}
%Our framework provides an efficient way to calculate eigenvectors of block tridiagonal matrices given the corresponding eigenvalues. However, in the most general case, our framework still requires the factorization of the characteristic polynomial, which can be a challenging problem for sufficiently large matrices. For practical purposes, a more efficient way of factoring the resulting polynomial is required that takes a better advantage of the recursive structure of the generated matrix polynomials. In addition, the recursive structure of matrix polynomials can potentially be used to construct fast eigenvector expansion algorithms for arbitrary block tridiagonal matrices by generalizing the approaches in~\cite{Driscoll:97,Sandryhaila:11a} from scalar polynomials to matrix polynomials.

% BiBTeX files (here: bibl_conf). The IEEEbib.bst bibliography
\bibliographystyle{siam}
\begin{small}
\bibliography{../../references}
\end{small}
\end{document}